\documentclass[12pt]{amsart}
\usepackage{amsmath}
\usepackage{amssymb}
\usepackage{amsfonts}
\usepackage{mathrsfs}
\usepackage{fullpage}

\newcommand{\g}{\mathfrak{g}}
\newcommand{\h}{\mathfrak{h}}
\newcommand{\mfk}{\mathfrak{k}}
\newcommand{\mfp}{\mathfrak{p}}

\newcommand{\mfb}{\mathfrak{b}}
\newcommand{\mfn}{\mathfrak{n}}

\newcommand{\bbZ}{\mathbb{Z}}
\newcommand{\bbC}{\mathbb{C}}

\newcommand{\bbQ}{\mathbb{Q}}

\newcommand{\cO}{\mathcal{O}}
\newcommand{\cB}{\mathcal{B}}

\DeclareMathOperator{\sgn}{\mathrm{sgn}}

\newtheorem{theorem}{Theorem}[section]
\newtheorem{proposition}[theorem]{Proposition}
\newtheorem{lemma}[theorem]{Lemma}

\theoremstyle{definition}
\newtheorem{definition}[theorem]{Definition}
\newtheorem{notation}[theorem]{Notation}

\theoremstyle{remark}

\newtheorem{remark}[theorem]{Remark}
\subjclass{Primary 22E50, Secondary 05E10}

\begin{document}

\title[Character Multiplicity Inversion by Induction]{Classical and Signed 
Kazhdan-Lusztig Polynomials:  Character Multiplicity Inversion by Induction}

\begin{abstract}
The famous Kazhdan-Lusztig Conjecture of the 1970s states that the multiplicity 
of an irreducible composition factor of a Verma module can be computed by 
evaluating Kazhdan-Lusztig polynomials at 1.  Thus the character of a Verma 
module is a linear combination of characters of irreducible highest weight 
modules where the coefficients in the linear combination are Kazhdan-Lusztig 
polynomials evaluated at 1.  Kazhdan-Lusztig showed that inverting and 
writing the character of an irreducible highest weight module as a linear 
combination of characters of Verma modules, the coefficients in the linear 
combination are also Kazhdan-Lusztig polynomials evaluated at 1, up to a 
sign.  In this paper, we show how to prove Kazhdan-Lusztig's character 
multiplicity inversion formula by induction using coherent continuation 
functors.  Unitary representations may be identified by determining if 
characters and signature characters are the same.  The signature character 
of a Verma module may be written as a linear combination of signature 
characters of irreducible highest weight modules where the coefficients in 
the linear combination are signed Kazhdan-Lusztig polynomials evaluated at 1.  
An analogous argument by induction using coherent continuation functors proves 
an analogous multiplicity inversion formula for signature characters:  the 
signature character of an irreducible highest weight module is a linear 
combination of signature characters of Verma modules where the coefficients,
up to a sign, are also signed Kazhdan-Lusztig polynomials evaluated at 1.
\end{abstract}

\author{Wai~Ling~Yee}
\address{Department of Mathematics and Statistics \\ 
University of Windsor \\
Windsor, Ontario \\
CANADA}
\email{wlyee@uwindsor.ca}
\thanks{The author is grateful for the support from a Discovery 
Grant and UFA from NSERC, and NSF grants DMS-0554278 and DMS-0968275.}

\maketitle

\section{Introduction}
Finding the composition factor multiplicities for a Verma module was one of 
the premier open problems in representation theory in the 1970s 
and the early 1980s.  Kazhdan and Lusztig conjectured in their seminal 
paper \cite{KL} that the multiplicities equaled Kazhdan-Lusztig polynomials 
evaluated at $1$.  Jantzen's Conjecture provided even deeper insight into 
the structure of a Verma module:  the various coefficients of the polynomials 
gave multiplicities of an irreducible composition factor within particular 
levels of the Jantzen filtration of the Verma module.  Brylinski-Kashiwara and 
Beilinson-Bernstein independently proved the Kazhdan-Lusztig Conjecture.  
Subsequently, Beilinson-Bernstein proved Jantzen's Conjecture in \cite{BB}.

The most common approach to classifying unitary representations is to first 
identify representations which admit invariant Hermitian forms, compute the 
signatures of those forms, and then determine which forms are definite.  The 
signature of the Hermitian form may be stored as a signature character, and 
then determining unitarity is equivalent to testing if the character and the 
signature character are the same.  See \cite{W}, \cite{Y3}, \cite{Y}, and 
\cite{Y2} for work on signature characters for irreducible highest weight 
modules.  Because understanding the structure of each 
level of the Jantzen filtration is crucial to computing signatures of 
invariant Hermitian forms, signed Kazhdan-Lusztig polynomials were defined in 
\cite{Y}, which store signature information for each level of the Jantzen 
filtration in its coefficients rather than multiplicities.

Consider $\g$, a complex semisimple Lie algebra, and $\h$, a Cartan subalgebra 
of $\g$.  Choose positive roots, let $\rho$ be the half sum of the positive 
roots, and let $\lambda \in \h^*$ be antidominant.  Let $\g = \mfn \oplus 
\h \oplus \mfn^-$ be the triangular decomposition and $\mfb = \h \oplus \mfn$ 
the Borel subalgebra determined by the choice 
of positive roots and let $M(\mu) = U(\g) \otimes_{U(\mfb)} 
\bbC_{\mu - \rho}$.
Let $W_{\lambda}$ be the integral Weyl group and $w_\lambda^0$ its long 
element.  Given $x \in W_{\lambda}$, 
composition factors of $M(x \lambda )$ are of the form $L(y \lambda )$ where 
$y \in W_\lambda$ and $y \leq x$.  The multiplicity of $L(y \lambda )$ in 
$M(x \lambda )$ is given by
$$[M(x \lambda) : L( y \lambda )] = P_{w_\lambda^0 x, w_\lambda^0 y}(1)$$
from which it follows that
$$ ch\, M(x \lambda) = \sum_{y \leq x} P_{w_\lambda^0 x, w_\lambda^0 y} (1) 
\,ch\, L( y \lambda ).$$
Kazhdan and Lusztig's inversion formula indicates that 
$$ch\, L( x \lambda ) = \sum_{y \leq x} (-1)^{\ell(x) - \ell(y)} P_{y,x}(1) 
\,ch\, M(y \lambda ).$$

In \cite{Y}, signed Kazhdan-Lusztig polynomials were defined in order to 
compute signatures of invariant Hermitian forms.  Assume $\g$ and $\h$ to 
be complexifications of $\g_0 \supset \h_0$ and assume $\h_0$ to be stable 
under $\theta$, the Cartan involution of $\g_0$.  Then $M( \mu )$ admits an 
invariant Hermitian form if and only if $\theta( \mfb ) = \mfb$ and $\mu$ 
is imaginary.  Invariant Hermitian forms on $M( \mu )$ are unique up to a real 
scalar.  The canonical form is called the Shapovalov form.  Its radical is the 
unique maximal proper submodule of $M( \mu )$, and hence it descends to a 
canonical invariant Hermitian form on $L( \mu )$.  Assume the real form to be 
equal rank.  Then $\theta( \mfb ) = \mfb$ implies that all of the roots are 
imaginary.  Assume that $\lambda$ is imaginary and antidominant.  Then the 
signed Kazhdan-Lusztig polynomials satisfy for small $t > 0$:
$$ e^{w(\rho) t} ch_s \, M( x \lambda + w(-\rho) t ) = \sum_{y \leq x}
P^{\lambda,w}_{w_\lambda^0 x, w_\lambda^0 y}( 1 ) \,ch_s\, L(y \lambda ).$$
The inversion formula providing an expression for $ch_s \, L(x \lambda )$ 
as a linear combination of $ch_s \, M(y\lambda + w(-\rho)t)$ found in 
\cite{Y} is cumbersome.  Particularly in 
light of a close relationship proved in \cite{Y2} between classical and 
signed Kazhdan-Lusztig polynomials, we are able to prove by induction using 
coherent continuation functors in the equal rank case the following formula, 
which is analogous to the inversion formula for characters: \\
{\bf Main Theorem:}
$$ ch_s \, L( x \lambda ) =  e^{w_\lambda^0(\rho)t} \sum_{y \leq x} 
(-1)^{\ell(x) - \ell(y)} 
P^{\lambda, w_\lambda^0}_{y,x}(1)  ch_s \,  M( y \lambda + w_\lambda^0(-\rho)t).$$
Here, $\epsilon$ is the $\bbZ_2$-grading on the root lattice with 
$\epsilon(\mu)$ being equal to the parity of the number of non-compact roots 
in an expression for $\mu$ as a sum of roots.

The paper is structured as follows.

The first half of the paper deals with classical Kazhdan-Lusztig polynomials.
First,  we review the relationship between Kazhdan-Lusztig polynomials and the 
Jantzen filtration in section 2.  We then review the coherent continuation 
functor in section 3.  In section 4, we prove Kazhdan and Lusztig's character 
multiplicity inversion formula by induction using coherent continuation 
functors.  Just as Gabber and Joseph showed in \cite{GJ} that Jantzen's 
Conjecture and coherent continuation functors gave recursive formulas for 
computing Kazhdan-Lusztig polynomials, permitting the stepwise computation of 
Kazhdan-Lusztig polynomials, applying coherent continuation 
functors allows us to prove a character multiplicity inverison formula by 
induction.

The second half of the paper concerns signed Kazhdan-Lusztig polynomials.
In section 5, we review signature character theory for highest weight modules, 
the Jantzen filtration,  and the definition of signed Kazhdan-Lusztig 
polynomials.  In section 6, we discuss coherent continuation functors in the 
enhanced setting used to compute signature characters.  In section 7, we
apply the same induction argument to signature characters in the case where 
the real form is compact to obtain an inversion formula.  In section 8, we 
show how a simple change of basis argument provides an inversion formula 
for other equal rank real forms.

Section 9 contains a description of future work.

\subsection{Acknowledgements:}  David Vogan and Gregg Zuckerman discussed  
using the coherent continuation representation of the Weyl group to compute
characters of irreducible Harish-Chandra modules at the 2010 Atlas of Lie 
Groups workshop.  Building representations using coherent continuation is the 
philosophy used in this paper.

\section{Kazhdan-Lusztig Polynomials and Jantzen's Conjecture}
\label{KLJantzen}
We use the following setup in the first half of this paper:
\begin{notation}
\begin{itemize}
\item[-] $\g$ is a complex semisimple Lie algebra
\item[-] $\h$ is a Cartan subalgebra
\item[-] $\Delta(\g,\h)$ is the set of roots
\item[-] $\Delta^+(\g,\h)$ is a choice of positive roots
\item[-] $\rho$ is one half the sum of the positive roots
\item[-] $\Pi$ is the set of simple roots
\item[-] $\g = \mfn \oplus \h \oplus \mfn^-$ is the corresponding triangular 
	decomposition
\item[-] $\mfb = \h \oplus \mfn$ is the corresponding Borel subalgebra
\item[-] $\Lambda_r$ is the root lattice
\item[-] $\Lambda$ is the integral weight lattice
\item[-] $\lambda \in \h^*$ is antidominant
\item[-] $W_\lambda$ is the integral Weyl group and $w_\lambda^0$ its long 
	element
\item[-] $\Delta_\lambda : = \{ \alpha \in \Delta(\g,\h) : (\lambda, 
\alpha^\vee) \in \bbZ \}$.  It is a root system with Weyl group $W_\lambda$.
\item[-] $\Pi_\lambda$ is a set of simple roots for $\Delta_\lambda$ 
	determined by $\rho$
\item[-] Given $\mu \in \h^*$, $M( \mu ) := U( \g ) \otimes_{U(\mfb)} 
	\bbC_{\mu - \rho}$ is the Verma module of highest weight $\mu - \rho$.
\item[-] $v_{\mu - \rho}$ the canonical generator of $M( \mu )$.
\item[-] $L(\mu)$ is the quotient of $M(\mu)$ by its unique maximal proper 
	submodule.  It is an irreducible highest weight module.
\end{itemize}
\end{notation}

We begin by reviewing the Shapovalov form and the Jantzen filtration.
Let $\sigma$ be the unique involutive automorphism of $\g$ such that 
$$\sigma( X_\alpha ) = Y_\alpha, \qquad \sigma( Y_\alpha ) = X_\alpha, \qquad
\sigma( H_\alpha ) = H_\alpha$$
where $\alpha \in \Delta^+(\g,\h)$ and $X_\alpha \in \g_\alpha$, $Y_\alpha \in 
\g_{-\alpha}$, and $H_\alpha \in \h$ are Chevalley generators of $\g$.  Then 
there exists a bilinear form $( \cdot, \cdot ) : M(\mu) \times 
M(\mu) \to \bbC$ that is contravariant:  contravariant means that 
$$ ( x \cdot u, v ) = (u, \sigma( x ) v ) \qquad \forall \, x \in \g
\text{ and } u, v \in M(\mu).$$
Normalized so that $(v_{\mu - \rho}, v_{\mu - \rho}) = 1$, the form is unique 
and called the (classical) Shapovalov form.
We denote it by $( \cdot, \cdot )_{\mu}$.
This form exists for every Verma module.

\begin{definition}  Let $V$ be a finite-dimensional vector space and let 
$(\cdot, \cdot )_t$ be an analytic family of bilinear forms on $V$ where 
$t \in (-\delta, \delta)$ and the forms are non-degenerate for $t \neq 0$.  
The Jantzen filtration is defined to be:
$$V = V^{(0)} \supset V^{(1)} \supset \cdots \supset V^{(n)} = \{ 0 \}$$
where $v \in V^{(n)}$ if there is exists an analytic map $\gamma_v : 
(-\epsilon, \epsilon) \to V$ for some $\epsilon > 0$ such that:
\begin{enumerate}
\item $\gamma_v( 0) = v$ and 
\item for every $u \in V$, $( \gamma_v(t), u )_t$ vanishes at least to order 
$n$ at $t=0$.
\end{enumerate}
\end{definition}

Since $(h \cdot u, v )_\mu = (u, h \cdot v )_\mu$ for all $h \in \h$ and 
$u,v \in M(\mu )$, therefore the weight space decomposition is an orthogonal 
decomposition into finite-dimensional subspaces.  Under the canonical 
identification of any $M( \mu )$ with $U( \mfn^-)$, it makes sense to take 
an analytic path of highest weights equal to $\mu$ at $t=0$ and to consider 
the corresponding Jantzen filtration on $M( \mu )$.  It was proved by 
Barbasch in \cite{Ba} that the Jantzen filtration does not depend on the 
analytic path chosen.

Let $\lambda$ be antidominant and let $x \in W_\lambda$. It is well-known that
the $j^{\text{th}}$ level of the Jantzen filtration, $M(x\lambda)_j := 
M(x\lambda)^{(j)} / M(x\lambda)^{(j+1)}$, is semisimple.  It is a direct sum 
of irreducible highest weight modules of the form $L( y \lambda )$ where 
$y \in W_\lambda$ and $y \leq x$.  Jantzen's Conjecture gives the multiplicity 
of $L( y \lambda )$ in the $j^{\text{th}}$ level of the Jantzen filtration.

\begin{theorem}(\cite{BB}) Jantzen's Conjecture:  Let $\lambda$ be 
antidominant and $x, y \in W_\lambda$.  Then:
$$[M(x\lambda)_j : L(y\lambda)] = \text{coefficient of } q^{(\ell(x) - \ell(y)
-j)/2} \text{ in } P_{w_\lambda^0 x, w_\lambda^0 y}(q).$$
\end{theorem} 

To compute Kazhdan-Lusztig polynomials, one may use the following formulas:
\begin{theorem} \label{KLRecursiveFormulas}
Let $\lambda$ be antidominant.
Kazhdan-Lusztig polynomials may be defined by $P_{x,x}(q) = 1$ for all 
$x \in W_\lambda$ and $P_{x,y}(q) = 0$ for $x, y \in W_\lambda$ and $x \not 
\leq y$ along with the recursive formulas where $s$ is a simple reflection:
\begin{itemize}
\item[a)] $P_{x,y}(q) = P_{xs, y}(q)$ if $ys < y$ and $x, xs \leq y$
\item[a')] $P_{x,y}(q) = P_{sx, y}(q)$ if $sy < y$ and $x, sx \leq y$
\item[b)] If $y < ys$, then:
$$q^c P_{xs,y}(q) + q^{1-c} P_{x,y}(q) = \sum_{zs < z} a_{z,y,1}
q^{\frac{\ell(y)-\ell(z)+1}{2}} P_{x,z}(q) + P_{x,ys}(q)$$
where $c=1$ if $xs > x$ and $c=0$ if $xs < x$.
\end{itemize}
\end{theorem}

\section{Coherent Continuation Functors}
We begin by recalling the definition of Category $\cO$ and related concepts 
discussed in \cite{BGG}.
\begin{definition}
Category $\cO$ is the category of $\g$-modules $V$ satisfying:
\begin{enumerate}
\item $V = \oplus_{\mu \in \h^*} V_\mu$
\item $V$ is finitely-generated over $U( \g )$
\item $V$ is $\mfn$-finite:  $\dim U(\mfn) v < \infty$ for every $v \in V$.
\end{enumerate}
\end{definition}
Verma modules and irreducible highest weight modules form two different bases 
of the Grothendieck group of Category $\cO$.

If $z \in Z(\g)$, the centre of the universal enveloping algebra, then the 
centre acts by scalars on Verma modules, so there
is a map $\chi_\mu : Z(\g) \to \bbC$ such that
$z \cdot v = \chi_\mu(z) v$ for every $z \in Z(\g)$ and every $v \in M( 
\mu )$.  Specifically, consider the decomposition $U(\g) = U(\h) \oplus 
( U(\g) \mfn  + \mfn^- U(\g) )$ and let $pr : U(\g) \to U(\h)$ be projection 
with respect to that direct sum.  Then:
$$\chi_\mu( z ) = \mu( pr(z)) \qquad \forall \, z \in Z(\g).$$
$\chi_\mu$ is called a central character.  $\chi_\mu = \chi_\nu$ if and only 
if $\nu = w \mu$ for some element $w$ of the Weyl group.

\begin{definition}
Category $\cO$ may be decomposed by central character into what are called 
blocks:
$$\cO_\mu := \{ V \in \cO : \exists \, N \text{ such that } (z-\chi_\mu(z))^N 
\text{ annihilates } V \, \forall \, z \in Z(\g) \}.$$
\end{definition}
For example, $M( x \lambda )$ and $L( x \lambda )$ belong to $\cO_\lambda$ 
for all $x \in W$.  
\begin{definition}
Decomposing category $\cO$ by infinitesimal character:
$$ \cO = \bigoplus_{\chi_\mu} \cO_\mu.$$
By $Pr_\mu$, we mean projection onto $\cO_\mu$ under the above direct sum.  
Decomposition under the above direct sum is called primary decomposition.  
\end{definition}
It is known that primary decomposition is an orthogonal decomposition.

To define coherent continuation functors, we need Jantzen's translation 
functors which allow us to translate modules from one block of Category $\cO$ 
to another.
\begin{definition}  Choosing $F(\mu)$ the 
finite-dimensional representation of extremal integral weight $\mu$, then
\begin{eqnarray*}
T^{\lambda + \mu}_\lambda : \cO_\lambda &\to& \cO_{\lambda + \mu} \\
V &\mapsto & Pr_{\lambda + \mu} ( V \otimes F(\mu))
\end{eqnarray*}
is Jantzen's translation functor from $\cO_\lambda \to \cO_\mu$.
By projecting first onto $\cO_\lambda$, the translation functor may be 
applied to all of Category $\cO$:
\begin{eqnarray*}
T^{\lambda + \mu}_\lambda : \cO &\to& \cO_{\lambda + \mu} \\
V &\mapsto & Pr_{\lambda + \mu} ( Pr_\lambda(V) \otimes F(\mu)).
\end{eqnarray*}
\end{definition}

Now we recall the definition of coherent continuation functors:
\begin{definition}
Let $\lambda \in \h^*$ be antidominant regular and let $s = s_\alpha \in 
W_\lambda$ be a simple reflection.  There exists an integral weight 
$\nu_\alpha$ such that $\lambda - \nu_\alpha$ is antidominant and the only 
simple root $\beta$ for which $(\lambda - \nu_\alpha,\beta) = 0$ is 
$\beta = \alpha$.
Then:
\begin{enumerate}
\item $T_\lambda^{\lambda - \nu_\alpha}$ is translation to the $\alpha$-wall,
\item $T_{\lambda - \nu_\alpha}^\lambda$ is translation from the $\alpha$-wall,
and
\item $\theta_\alpha = T_{\lambda-\nu_\alpha}^\lambda \circ T_\lambda^{\lambda
- \nu_\alpha}$ is the coherent continuation functor across the 
$\alpha$-wall or the reflection functor across the $\alpha$-wall.
\end{enumerate}
\end{definition}

We recall the following facts about the coherent-continuation functor:
\begin{theorem} \label{CCTheorem}
Let $\lambda$ be antidominant regular, $\alpha \in 
\Pi_\lambda$, and let $s = s_\alpha$.  If $x \in W_\lambda$, then:
\begin{enumerate}
\item \label{CCCoherent}
$\theta_\alpha M( x \lambda ) = \theta_\alpha M( xs \lambda )$
\item If $x < xs$, there is a short exact sequence
$$ 0 \to M( xs\lambda ) \to \theta_\alpha M( x \lambda ) \to M( x \lambda )
\to 0.$$
\item If $x > xs$, then $\theta_\alpha L( x \alpha ) = 0$.
\item If $x < xs$, then $\theta_\alpha L( x \alpha ) \neq 0$ and it has a 
unique simple quotient and a unique simple submodule, both isomorphic to 
$L( x \lambda )$.
\item If $x < xs$, then there is a chain complex 
$$ 0 \to L( x \lambda ) \hookrightarrow \theta_\alpha L( x \lambda ) 
\twoheadrightarrow L( x \lambda ) \to 0$$
whose cohomology is denoted by $U_\alpha L( x \lambda )$.  Letting 
$a_{w_\lambda^0 x, w_\lambda^0 y, j}$ denote the coefficient of $q^{\frac{
\ell(x) - \ell(y)-j}{2}}$ in $P_{w_\lambda^0 x, w_\lambda^0 y}(q)$, 
$$ch\, U_\alpha L(x \lambda ) = ch\, L( xs \lambda ) + \sum_{ y \in W_\lambda | 
y > ys} a_{w_\lambda^0 x, w_\lambda^0 y, 1} \, ch\, L( y \lambda ).$$
\end{enumerate}
\end{theorem}
These are the formulas we will use to prove the character multiplicity 
inversion formula by induction.

\section{Kazhdan-Lusztig Polynomials:  Character Multiplicity Inversion}
The goal of this section is to prove Kazhdan-Lusztig's character multiplicity 
inversion formula by induction:
\begin{theorem}(\cite{KL})
Let $\lambda \in h^*$ be antidominant and let $x \in W_\lambda$.  Then:
$$ch\, L( x \lambda ) = \sum_{y \leq x} (-1)^{\ell(x) - \ell(y)} P_{y,x}(1) 
\,ch\, M(y \lambda ).$$
\end{theorem}
\begin{proof}
We prove this by induction on $x$.  When $x=1$, $M( \lambda ) = L( \lambda )$ 
while $P_{1,1}(q) = 1$, so the formula holds.

We suppose now by induction that for all $z < x$, 
$$ch\, L( z \lambda ) = \sum_{w \leq z} (-1)^{\ell(z) - \ell(w)} \, P_{w,z}(1)
\, ch \, M( w \lambda ).$$
In particular, for some $s = s_\alpha$ simple, $x > xs$.  By induction, 
$$ ch \, L( xs \lambda ) = \sum_{ y \leq xs} (-1)^{\ell(xs) - \ell(y)}
\, P_{y,xs}(1) \, ch \, M(y \lambda).$$
Following an idea of Gabber and Joseph, who showed how to build recursive 
formulas for computing Kazhdan-Lusztig polynomials using coherent continuation 
functors assuming that Jantzen's Conjecture holds,
we apply $\theta_\alpha$ to both sides of this equation and get:
$$ch \, \theta_\alpha L( xs \lambda ) = \sum_{y \leq xs} (-1)^{\ell(xs) - 
\ell(y)} \, P_{y,xs}(1) \, ch \, \theta_\alpha M( y \lambda ).$$ 
The equations of Theorem \ref{CCTheorem} now imply:
$$2 \, ch \, L( xs \lambda ) + ch \, L( x \lambda )  + 
\sum_{z > zs} a_{w_\lambda^0 xs, w_\lambda^0 z, 1} \, ch L( z \lambda )
= \sum_{y \leq xs} (-1)^{\ell(xs) - \ell(y)} \, P_{y,xs}(1) \left( 
ch\, M( y \lambda) + ch \, M( ys \lambda ) \right).$$
Applying the induction hypothesis to $ch \, L( xs \lambda )$ and $ch \, 
L( z \lambda)$ and rearranging, we see that
\begin{eqnarray*}
ch \, L( x \lambda ) &=&  \sum_{y \leq xs} (-1)^{\ell(xs) - \ell(y)} \, 
P_{y,xs}(1) \left( - ch \, M(y\lambda) + ch\, M( ys \lambda ) \right) \\
&&- \sum_{z > zs} a_{w_\lambda^0 xs, w_\lambda^0 z, 1} \sum_{w \leq z}
(-1)^{\ell(z) - \ell(w)} \,P_{w,z}(1) \, ch \, M( w \lambda )
\end{eqnarray*}
To prove the theorem, we study the coefficient in front of a fixed 
$ch \, M( w \lambda )$ where $w \leq x$.  Let $c( M(w \lambda))$ denote this 
coefficient.

Since $x > xs$, by Property Z, $x \geq w, ws$.

{\bf Case 1):}  Both $w, ws \leq xs$. \\
The coefficient of $M( w \lambda )$ is:
\begin{eqnarray*}
c( M(w \lambda)) &=& (-1)^{\ell(xs) - \ell(w)} P_{w,xs}(1) (-1) + 
(-1)^{\ell(xs) - \ell(ws)} P_{ws, xs}(1) \\
&&- \sum_{z > zs, z \geq w}
a_{w_\lambda^0 xs, w_\lambda^0 z, 1} (-1)^{\ell(z) - \ell(w)} P_{w,z}(1) \\
&=& (-1)^{\ell(x)-\ell(w)} \left( P_{w,xs}(1) + P_{ws, xs}(1) \right) - 
\sum_{z > zs, z \geq w} a_{w_\lambda^0 xs, w_\lambda^0 z, 1} (-1)^{\ell(z) -
\ell(w)} P_{w,z}(1).
\end{eqnarray*} 
Now $a_{w_\lambda^0 xs, w_\lambda^0 z, 1}$ is the coefficient of 
$q^{\frac{\ell(xs) - \ell(z) - 1}{2}}$ in $P_{w_\lambda^0 xs, w_\lambda^0 
z}$, and thus $\ell(x)$ and $\ell(z)$ have the same parity if $a_{w_\lambda^0
xs, w_\lambda^0 z, 1} \neq 0$.  We may replace $(-1)^{\ell(z) - \ell(w)}$ 
with $(-1)^{\ell(x) - \ell(w)}$.  Furthermore, it is well-known that 
$a_{w_\lambda^0 xs, w_\lambda^0 z, 1} = a_{z, xs, 1}$ (see, for example, 
\cite{H}, p. 165).
Thus:
\begin{eqnarray*}
c( M( w \lambda ) ) &=& (-1)^{\ell(x) - \ell(w)} \left( P_{w,xs}(1) + 
P_{ws, xs}(1) - \sum_{z > zs, z \geq w} a_{z,xs,1} P_{w,z}(1) \right) \\
&=& (-1)^{\ell(x)-\ell(w)}  P_{w,x}(1)
\end{eqnarray*}
by formula b) of Theorem \ref{KLRecursiveFormulas}.

{\bf Case 2):}  $x \geq w$ but $xs \not \geq w$: \\
Recall by Property Z that both $w, ws \leq x$.  Also, we must have $w > ws$ 
and $xs \geq ws$
by Property Z.  The coefficient of $M( w \lambda )$ is:
\begin{eqnarray*}
c( M(w\lambda)) &=& (-1)^{\ell(xs) - \ell(ws)} P_{ws, xs}(1) - \sum_{z > zs, 
z \geq w} a_{w_\lambda^0 xs, w_\lambda^0 z, 1} (-1)^{\ell(z) - \ell(w)} 
P_{w,z}(1) \\
&=& (-1)^{\ell(x) - \ell(w)} \left( P_{ws,xs}(1) - \sum_{z > zs, z \geq w}
a_{w_\lambda^0 xs, w_\lambda^0 z, 1} P_{w,z}(1) \right ) \\
&=& (-1)^{\ell(x) - \ell(w)} \left( P_{ws,xs}(1) \right)
\end{eqnarray*}
since $a_{w_\lambda^0 xs, w_\lambda^0 z, 1} \neq 0$ implies that $z \leq xs$ 
while $P_{w,z} \neq 0$ implies that $w \leq z$.  There are no $z$ so that 
$w \leq z \leq xs$, so the sum is zero.

It is known that the right side is now $(-1)^{\ell(x) - \ell(w)} P_{w,x}(1)$ 
(see, for example, \cite{dC} formula 4.1 c)).

{\bf Case 3):}  $xs \geq w$ but $xs \not \geq ws$: \\
In this case, $w < ws \leq x$ by Property Z.  The coefficient of $M( w\lambda)$
is:
\begin{eqnarray*}
c( M(w\lambda)) &=& -(-1))^{\ell(xs) - \ell(w)} P_{w,xs}(1) - \sum_{z > zs,
z \geq w} (-1)^{\ell(z) - \ell (w)} a_{w_\lambda^0 xs, w_\lambda^0 z,1} 
P_{w,z}(1) \\
&=& (-1)^{\ell(x)-\ell(w)}\left( P_{w,xs}(1) + P_{ws,xs}(1) - \sum_{z>zs, z 
\geq w} a_{z,xs,1} P_{w,z}(1) \right) \quad \text{since } P_{xs,ws} = 0 \\
&=& (-1)^{\ell(x)-\ell(w)} P_{w,x}(1).
\end{eqnarray*}

In all other cases, the coefficient is $0$.

We see that in all cases,
$$ch \, L( x\lambda ) = \sum_{y \leq x} (-1)^{\ell(x) - \ell(y)} ch \, M( y 
\lambda ),$$
proving Kazhdan-Lusztig's character multiplicity inversion formula by 
induction.
\end{proof}

\section{Signed Kazhdan-Lusztig Polynomials and Jantzen's Conjecture}
We specialize the setup from the first half of the paper for the second half 
of this paper:
\begin{notation}
\begin{itemize}
\item[-] $\g_0$ is an equal rank real semisimple Lie algebra 
\item[-] $\theta$ is a Cartan involution of $\g_0$
\item[-] $\g_0 = \mfk_0 \oplus \mfp_0$ is the corresponding Cartan decomposition
\item[-] $\h_0$ is a $\theta$-stable Cartan subalgebra of $\g_0$
\item[-] drop the subscript 0 to denote complexification
\end{itemize}
\end{notation}

A Hermitian form $\left< \cdot, \cdot \right>$ on a $\g$-module $V$ is 
invariant if 
$$\left< X \cdot u, v \right> + \left< u, \bar{X} \cdot v \right> = 0$$
for all $X \in \g$, where complex conjugation is with respect to the real 
form $\g_0$ of $\g$.

In our setup, the Verma module $M( \lambda )$ accepts a non-degenerate 
invariant Hermitian form if and only if $\theta( \mfb ) = \mfb$ and $\lambda$
is imaginary (\cite{Y3}, p. 641).  Since $\g$ is equal rank, therefore $\h$ is 
compact and all of the roots are imaginary.  The form is unique up to a real 
scalar.
If it is normalized so that $\left< v_{\lambda-\rho}, v_{\lambda-\rho} 
\right> =1$, then it is called the Shapovalov (Hermitian) form and denoted 
by $\left< \cdot, \cdot \right>_\lambda$.

Since the Chevalley basis vectors satisfy $-\bar{X}_\alpha = Y_\alpha$ if 
$\alpha$ is compact while $-\bar{X}_\alpha = -Y_\alpha$ if $\alpha$ is 
non-compact (see \cite{Y3}, 5.2.17), therefore $\left< \cdot, \cdot 
\right>_\lambda$ and $( \cdot, \cdot )_\lambda$, the invariant Shapovalov 
Hermitian form and the contravariant Shapovalov 
bilinear form, are related.  The weight space decomposition is an orthogonal 
decomposition into finite dimensional subspaces (see \cite{Y3}, p. 643) and 
just as in section 
\ref{KLJantzen} we can discuss the signature of the Shapovalov form.  We 
can also discuss the Jantzen filtration on that Verma module 
by considering an analytic path of imaginary highest weights.  The Jantzen 
filtration of a Verma module arising from invariant Hermitian forms and the 
Jantzen filtration of a Verma module arising from contravariant forms are the 
same.  We denote the levels of the filtration with 
$M(\lambda)^{\left< j \right>}$ or with 
$M(\lambda)^{(j)}$ interchangeably.  Furthermore, the Shapovalov form 
descends to an invariant Hermitian form on $L(\lambda)$.  
We can also discuss the signature for the highest weight module $L( \lambda )$.
The form is non-degenerate, and signatures for each of the weight spaces can be 
stored in a signature character:

\begin{definition}
The signature character of a non-degenerate invariant Hermitian form 
$\left< \cdot, \cdot \right>$ on a  
highest weight module $V$ whose weights are $\Delta(V)$ is
$$ch_s \left< \cdot, \cdot \right> = ch_s V = \sum_{\mu \in \Delta(V)} (p(\mu) 
- q(\mu)) e^\mu$$
where the signature on $V_\mu$ is $(p(\mu), q(\mu))$.
\end{definition}

The general approach to computing signature characters has been as follows.  
The Shapovalov form on a Verma module is non-degenerate if and only if the 
Verma module is irreducible.  $M( \lambda )$ is reducible if and only if 
$(\lambda, \alpha^\vee) \in \bbZ^+$ for some positive root $\alpha$, so 
$M( \lambda )$ is irreducible away from various affine hyperplanes.   Within 
any region bounded by these so called reducibility hyperplanes, the Shapovalov 
form cannot become degenerate, hence the signature is constant.  The 
antidominant Weyl chamber does not contain any reducibility hyperplanes.  
Within the reducibility hyperplane-free region containing the antidominant 
Weyl chamber, an asymptotic argument gives the signature character 
(see \cite{W}).  In other regions, an induction argument and combinatorial 
analysis to determine how signatures change as you cross reducibility 
hyperplanes to travel from one region to another give the signature 
character for other irreducible Verma modules (see \cite{Y}).  If $M( \lambda )$
is reducible, we wish instead to compute the signature of $L( \lambda )$.  Such
formulas were determined in \cite{Y} and use so-called signed Kazhdan-Lusztig 
polynomials.  We review the definition of signed Kazhdan-Lusztig polynomials.

Let $\lambda$ be antidominant and imaginary and let $x \in W_\lambda$.  The 
$j^{\text{th}}$ level of the Jantzen filtration of $M(x\lambda)$ is semisimple 
and equal to the direct sum of modules of the form $L(y\lambda)$ where $y \in 
W_\lambda$  and $y \leq x$.  While Kazhdan-Lusztig polynomials record $j^{
\text{th}}$ level multiplicities in their coefficients, signed Kazhdan-Lusztig 
polynomials were defined to record signatures.  First recall that the 
Jantzen filtration naturally provides forms on each filtration level:
\begin{definition} 
Let $\lambda$ be antidominant, let $x \in W_\lambda$, and let $\delta \in \h^*$
be regular and imaginary.  There is 
an invariant Hermitian form on $M( x \lambda)^{\left< j \right>}$:
$$ \left< u, v \right>_j 
:= \lim_{t \to 0^+} \frac{1}{t^j} \left< \gamma_u(t), \gamma_v(t) 
\right>_{x \lambda + \delta t} 
\qquad \forall \, u, v \in M(x\lambda).$$
This form descends to a non-degenerate form on $M( x \lambda )_j = 
M( x \lambda)^{\left< j \right>} / M(x \lambda)^{\left<j+1\right>}$
which we will also denote by $\left< \cdot, \cdot \right>_j$.  
\end{definition}
\begin{remark}
Note that the signatures of the forms $\left< \cdot, \cdot \right>_j$ 
depend on the direction $\delta$ of the filtration.
For any two directions belonging to the interior of the same Weyl chamber, 
there are no reducibility hyperplanes separating $x\lambda$ plus the direction 
vectors, so the signatures of $\left< \cdot, \cdot \right>_j$ are the same for 
those two choices of direction vector.  Without loss of generality, we take 
our direction to be some $w(-\rho)$ where $w \in W$.
\end{remark}

\begin{definition} 
Let $\lambda$ be antidominant, let $x \in W_\lambda$, and $w \in W$.  Let the 
Jantzen filtration on $M( x \lambda )$ arise from the analytic family of 
invariant Hermitian forms $\left< \cdot, \cdot \right>_{x\lambda + w(-\rho)t}$.
Recall the non-degenerate invariant Hermitian form $\left< \cdot, \cdot 
\right>_j$ defined on $M( x \lambda )_j$.
Let $y \in W_\lambda$.  Recall that $M(x\lambda)_j$ is a direct sum of 
irreducible highest weight modules of the form $L( y \lambda )$.  The 
constants $a^{\lambda,w}_{w_\lambda^0 x, w_\lambda^0 y, j}$ are defined 
to satisfy:
$$ch_s \left< \cdot, \cdot \right>_j = \sum_{y \leq x} a^{\lambda,w}_{
w_\lambda^0 x, w_\lambda^0 y, j} ch_s L( y \lambda )$$
where by $ch_s L( y \lambda )$ we mean the signature of the Shapovalov form.
Signed Kazhdan-Lusztig polynomials are defined by:
$$ P^{\lambda,w}_{w_\lambda^0 x, w_\lambda^0 y}(q) = \sum_{j \geq 0}
a^{\lambda,w}_{w_\lambda^0 x, w_\lambda^0 y,j} 
q^{\frac{\ell(x)-\ell(y)-j}{2}}.$$
\end{definition}

By Proposition 3.3 of \cite{V}, the signature character for small $t > 0$ is 
equal to the sum $\sum_j ch_s \left< \cdot, \cdot \right>_j$ of all 
the $j{\text{th}}$-level signature characters.
Therefore:
\begin{proposition}\label{SKLInterpretation}
 For small $t > 0$:
$$e^{w\rho t} \, ch_s \, M( x \lambda + w(-\rho)t ) = \sum_j ch_s \, 
\left< \cdot, \cdot \right>_j = \sum_{ y \leq x}
P^{\lambda,w}_{w_\lambda^0 x, w_\lambda^0 y}(1) ch_s L(y \lambda ).$$
\end{proposition}
The left hand side is known by work of \cite{Y3} while we would like to 
obtain a formula for $ch_s \, L( x \lambda )$.
In \cite{Y}, this formula was inverted to give a formula for 
$ch_s L(x \lambda)$ as a sum of signature characters of irreducible Verma 
modules.  The formula was unwieldy and required two improvements:
\begin{enumerate}
\item  The coefficients in the linear combination required simplification. We 
prove in this paper that the coefficients are, up to a sign, signed 
Kazhdan-Lusztig polynomials evaluated at $1$.
\item Signed Kazhdan-Lusztig polynomials are equal to classical Kazhdan-Lusztig 
polynomials evaluated at $-q$ rather than $q$ and multiplied by a sign 
(the main theorem of \cite{Y2}).
\end{enumerate}

The following recursive formulas are satisfied by signed Kazdhan-Lusztig 
polynomials:
\begin{proposition} \label{SKLRecursive} 
(Theorem 4.6.10 of \cite{Y}, corrected in the proof of 
\cite{Y2} Theorem 4.6)
Let $\lambda \in \h^*$ be antidominant.
Signed Kazhdan-Lusztig polynomials may be computed using
$P_{x,x}^{\lambda,w_\lambda^0}(q) = 1$, $P^{\lambda,w_\lambda^0}_{x,y}(q) = 0$
if $x \not \leq y$, and the following recursive formulas where $s = s_\alpha$
where $\alpha \in \Pi_\lambda$:
\begin{itemize}
\item[a)]  If $ys > y$ and $xs > x \geq y$ then:
$$P^{\lambda,w_\lambda^0}_{w_\lambda^0 x, w_\lambda^0 y}(q) =
(-1)^{\epsilon( (\lambda,\alpha^\vee) x \alpha )} P^{\lambda,
w_\lambda^0}_{w_\lambda^0
xs, w_\lambda^0 y} (q)
= (-1)^{\epsilon((x\lambda, x \alpha^\vee ) x \alpha)} P^{\lambda,w_\lambda^0}_{
w_\lambda^0 xs, w_\lambda^0 y}(q).$$

\item[a')] If $sy > y$ and $sx > x \geq y$ then:
$$P^{\lambda,w_\lambda^0}_{w_\lambda^0 x, w_\lambda^0 y}(q) = (-1)^{\epsilon(
(x\lambda, \alpha^\vee) \alpha)} P^{\lambda,w_\lambda^0}_{w_\lambda^0 sx,
w_\lambda^0 y}(q).$$

\item[b)] If $y > ys$ and $x < xs$ and $x > y$ then:
\begin{eqnarray*}
 - (-1)^{\epsilon(( x \lambda, x \alpha^\vee ) x \alpha)} P^{\lambda,
w_\lambda^0}_{ w_\lambda^0 xs, w_\lambda^0 y}(q) + q P^{\lambda,
w_\lambda^0}_{w_\lambda^0 x,
w_\lambda^0 y} (q) &=& q \sum_{z \in W_\lambda | z < zs} a^{\lambda,
w_\lambda^0}_{w_\lambda^0 z, w_\lambda^0 y} q^{\frac{\ell(z) - \ell(y) - 1}{2}}
P^{\lambda,w_\lambda^0}_{ w_\lambda^0 x,
w_\lambda^0 z}(q) \\
&& -(-1)^{\epsilon( (ys \lambda, ys \alpha^\vee) ys\alpha)}
P^{\lambda,w_\lambda^0}_{w_\lambda^0 x, w_\lambda^0 ys}(q).
\end{eqnarray*}
\end{itemize}
\end{proposition}

\section{Gabber and Joseph's Enhanced Setting for Coherent Continuation}
Recall that Gabber and Joseph showed in \cite{GJ} how to use coherent 
continuation functors and Jantzen's Conjecture to prove recursive formulas for 
Kazhdan-Lusztig polynomials.  In order to study how coherent continuation 
affects signature characters and to study signed Kazhdan-Lusztig polynomials,
we need to use Gabber and Joseph's generalization of Category $\cO$.

Our setup for this section is as follows:
\begin{notation}
\begin{itemize}
\item[-] $A = \bbC[t]_{(t)}$, a local ring
\item[-] $\cB := \{ X_\alpha, Y_\alpha : \alpha \in \Delta^+(\g,\h)\} \cup 
	\{ H_\alpha : \alpha \in \Pi \}$ is the Chevalley basis for $\g$
\item[-] $\g_\bbZ = \text{span}_\bbZ \cB$ 
\item[-] $\h_\bbZ$, $\mfb_\bbZ$, $\mfn_\bbZ$ etc. are defined similarly
\item[-] $\g_A := \g_\bbZ \otimes_\bbZ A$
\item[-] $\h_A$, $\mfb_A$, $\mfn_A$ etc. are defined similarly
\item[-] Given $\lambda \in \h_A^*$, $A_\lambda$ is the one-dimensional 
	$\h_A$-module where $h \in \h$ acts by $\lambda(h)$
\item[-] Given $\lambda \in \h_A^*$, $M( \lambda )_A = U( \g_A ) 
	\otimes_{\mfb_A} A_{\lambda-\rho}$
\end{itemize}
\end{notation}

Gabber and Joseph make the following generalization of Category $\cO$:
\begin{definition}
Let $\lambda \in \h_A^*$ and let $C \subset \h_A^*$ be of the form 
$\lambda + \Lambda$, where $\Lambda$ is the integral weight lattice.  Then 
category $K_C$ is the subcategory of $U(\g_A)$ modules $V$ such that:
\begin{enumerate}
\item $V = \sum_{\mu \in C - \rho} V_\mu$
\item $V$ is $U( \mfn_A )$-finite (i.e. $\dim_A U(\mfn_A) \cdot v < \infty$ 
	for all $v \in V$)
\item $V$ is finitely generated over $U( \g_A)$
\end{enumerate}
\end{definition}

The Verma modules $M( \lambda )_A$ are objects in $K_C$.  Each $M( \lambda )_A$ 
has a unique maximal proper submodule.  In general, for every maximal ideal 
$m$ of $A$, there is  a unique maximal proper submodule of $M( \lambda )_A$ 
containing $m M(\lambda)_A$.  The corresponding simple quotient is denoted by 
$L( m, \lambda )$.  We have chosen $A$ to be a local ring, so the maximal 
ideal of $A$ is unique and the maximal proper submodule of $M( \lambda )_A$ 
is unique, and we may write $L( \lambda )_A$ for the corresponding simple 
quotient without ambiguity.  These $L( \lambda)_A$'s are the simple objects 
of category $K_C$.

\begin{notation}
Recall that we are working over the local ring $A = \bbC[t]_{(t)}$.
We denote by $\bar{\cdot}$ specialization at $t=0$.  That is, given any 
$A$-module $V$, $\bar{V} = V / tV$.  So, given
$\lambda + \delta t \in \h_A^*$, where $\lambda, \delta \in \h^*$:
\begin{itemize}
\item[-] $\overline{\lambda+\delta t} = \lambda$,
\item[-] $\overline{M(\lambda+\delta t)_A} = M( \lambda )$, and
\item[-] $\overline{L(\lambda+\delta t)_A} = L( \lambda )$.
\end{itemize}
\end{notation}

\begin{definition}
If $D \subset C$, then category $K_D$ is the subcategory of $K_C$ whose 
simple subquotients are $L( \lambda )_A$'s where $\lambda \in D$.
\end{definition}

\begin{definition}
In the context of category $K_C$, a subset $D \subset C$  is a block if 
there is some $\mu \in \h_A^*$ and $\delta \in \h^*$ such that
$D = W_{\bar{\mu}} \mu + \delta t$.  Given a block $D$, 
the following definitions may be made:
\begin{enumerate}
\item $J_D := \cap_{\mu \in D} \ker \chi_\mu$
\item For $V \in \mathrm{Ob} K_C$, $Pr_D V := \{ v \in V : \forall \, z \in 
	J_D, \exists \,  N \in \bbZ^+ \text{ such that } z^n v = 0 \}$.
\item $V = \bigoplus_{D_i} Pr_{D_i} V$ is the primary decomposition of $V$.
\item Let $K_D$ be the subcategory of $K_C$ such that $V$ is an object in 
	$K_D$ if and only if $J_D^N V = 0$ for large enough $N$.  $Pr_D$ 
	takes objects in $K_C$ to objects in $K_D$.
\end{enumerate}
\end{definition}

Gabber and Joseph made the following extension of Jantzen's translation 
functors to category $K_C$:
\begin{definition}
Given $\lambda \in \h^*$ antidominant and regular and $\delta \in 
\h^*$ regular, let $D = W_\lambda \lambda + \delta t$.  Given  $\mu \in 
\Lambda$, let $D' = W_\lambda( \lambda + \mu) + \delta t$ and let $F(\mu)_A$ 
denote the finite-dimensional $\g_A$-module with extremal weight $\mu$.  Then 
$$T_D^{D'} V = Pr_{D'}( F( \mu )_A \otimes_A ( Pr_D V) )$$
is Jantzen's translation functor from $D$ to $D'$.
\end{definition}

This leads to their extension of the defintion of coherent continuation 
functors:
\begin{definition}
Let $\lambda \in \h^*$ be antidominant regular, $\delta \in \h^*$ be regular,
$D = W_\lambda \lambda + \delta t$, and let $s_\alpha \in 
W_\lambda$ be a simple reflection.  There exists an integral weight 
$\nu_\alpha$ such that $\lambda - \nu_\alpha$ is antidominant and the only 
simple root $\beta$ for which $(\lambda - \nu_\alpha,\beta) = 0$ is 
$\beta = \alpha$.
Let $D_\alpha = W_\lambda( \lambda -\nu_\alpha) + \delta t$.
Then:
\begin{enumerate}
\item $T_D^{D_\alpha}$ is translation to the $\alpha$-wall,
\item $T_{D_\alpha}^D$ is translation from the $\alpha$-wall,
and
\item $\theta_\alpha = T_{D_\alpha}^D \circ T_D^{D_\alpha}$ is the 
coherent continuation functor across the 
$\alpha$-wall or the reflection functor across the $\alpha$-wall.
\end{enumerate}
\end{definition}

\begin{notation}
Given $\lambda \in \h^*$ regular antidominant, $x \in W_\lambda$, and 
$\delta \in \h^*$ regular, by $ch_s \overline{ M( x \lambda + \delta t)}$  or
$ch_s \overline{ L( x \lambda + \delta t)}$
we mean the limiting signature on the specialized module as $t \to 0^+$.
\end{notation}
For example, we may rewrite Proposition \ref{SKLInterpretation} as
$$ch_s \overline{M( x \lambda + w(-\rho) t )} = \sum_{y \leq x} P^{\lambda,
w}_{w_\lambda^0 x, w_\lambda^0 y} ch_s L( y \lambda ).$$

We have the following results for signature characters and coherent 
continuation:
\begin{theorem} \label{SKLCCTheorem}
Let $\lambda \in \h^*$ be antidominant regular,
$\alpha \in \Pi_\lambda$, and let $s = s_\alpha$.  Let $D = W_\lambda \lambda 
+ w_\lambda^0(-\rho) t$ and let $D_\alpha = W_\lambda( \lambda - \nu_\alpha) + 
w_\lambda^0(-\rho)t$ be a block so that $T_D^{D_\alpha}$ is a translation 
functor to the 
$\alpha$-wall.  Let $\lambda_\alpha^+$ denote the highest weight of 
$F( \nu_\alpha )$ and $\lambda_\alpha^-$ denote the highest weigt of 
$F( - \nu_\alpha )$.  Then:
\begin{enumerate}
\item \label{SCoherent} If $z \in W_\lambda$ and $z < zs$, then:
\begin{eqnarray*}
ch_s \, \overline{ \theta_\alpha M( z \lambda + w_\lambda^0(-\rho) t)} &=&
-(-1)^{\epsilon((\lambda,\alpha^\vee)z\alpha)} 
ch_s \, \overline{ \theta_\alpha M( zs \lambda + w_\lambda^0(-\rho) t)}  \\
&=& 
(-1)^{\epsilon(\lambda_\alpha^+ + \lambda_\alpha^- + z( \lambda, \alpha^\vee)
\alpha)} ch_s \,\overline{M( zs \lambda + w_\lambda^0(-\rho) t)}  \\
&&- (-1)^{\epsilon( \lambda_\alpha^+ + \lambda_\alpha^-)} 
ch_s \, \overline{M( z \lambda + w_\lambda^0(-\rho) t)}
\end{eqnarray*}
\item \label{SCoherentVanish} 
If $z > zs$, then $ch_s \theta_\alpha L( z \lambda ) = 0$.
\item \label{SCoherentSimple}
If $z < zs$, then
\begin{eqnarray*}
ch_s \, \overline{\theta_\alpha L( z \lambda + w_\lambda^0 (-\rho) t )} &=& 
(-1)^{\epsilon(\lambda_\alpha^+ + \lambda_\alpha^- + z(\lambda,\alpha^\vee)
\alpha)} ch_s \, L( zs \lambda ) \\
&&  
-(-1)^{\epsilon(\lambda_\alpha^+ + \lambda_\alpha^-)}
\sum_{y > ys} a^{\lambda,w_\lambda^0}_{w_\lambda^0 z, w_\lambda^0 y, 1}
ch_s L( y \lambda ).
\end{eqnarray*}
\end{enumerate}
\begin{proof}
To prove \eqref{SCoherent}, we begin by proving that if $z < zs$,
\begin{eqnarray*}
ch_s \,\overline{T_D^{D_\alpha} M( z \lambda + w_\lambda^0(-\rho) t )} &=&
- (-1)^{\epsilon((\lambda,\alpha^\vee)z\alpha)}
ch_s \,\overline{T_D^{D_\alpha} M( zs \lambda + w_\lambda^0(-\rho) t )} \\
&=& (-1)^{\epsilon(\lambda_\alpha^- + z \nu_\alpha)} ch_s \, 
\overline{M( z( \lambda - \nu_\alpha) + w_\lambda^0(-\rho) t)}
\end{eqnarray*}
From the formulas at the bottom of p. 183 of \cite{Y},  for every $z$,
$$ch_s \overline{
T_D^{D_\alpha} M( z \lambda + w_\lambda^0( -\rho)t)} = 
\sgn( \bar{c}_z')
ch_s \, \overline{M( z( \lambda - \nu_\alpha ) + w_\lambda^0(-\rho) t)}.$$
The formula for $c_z'$ from the top of p. 184 of \cite{Y} is
$$c_z' = (-1)^{\epsilon(\lambda_\alpha^- + z \nu_\alpha)} \, D_{F(-\nu_\alpha)}
(-z\nu_\alpha) \, a_{-z \nu_\alpha}'.$$
Lemma 4.6.7 of \cite{Y} states that the second term in the product 
always has sign $1$ while Lemma 4.6.9.(i) of \cite{Y} states that the third 
term in the product has sign $1$ if $z < zs$.  Thus $\sgn( \bar{c}_z' ) = 
(-1)^{\epsilon( \lambda_\alpha^- + z \nu_\alpha)}$ if $z < zs$.  If $z > zs$, 
we must determine the sign of the third term, $\sgn( \bar{a}_{-z 
\nu_\alpha}' )$.  Follow the proof 
of 4.6.9.(i) exactly up to the last sentence.  Then note that since 
$(-z \nu_\alpha, z \alpha^\vee ) < 0$, therefore $\sgn(\bar{a}_{-z \nu_\alpha}')
 = -1$.  Thus  $\sgn( \bar{c}_z') = - (-1)^{\epsilon(\lambda_\alpha^- + 
z \nu_\alpha)}$ if $z > zs$; or if $z < zs$, $\sgn( \bar{c}_{zs}') = - 
(-1)^{\epsilon(\lambda_\alpha^- + 
zs \nu_\alpha)}$.  From this and from $z \nu_\alpha - zs \nu_\alpha = 
z( \lambda,\alpha^\vee) \alpha) \alpha$, it follows that 
$ch_s \, \overline{ \theta_\alpha M( z \lambda + w_\lambda^0(-\rho) t)} =
-(-1)^{\epsilon((\lambda,\alpha^\vee)z\alpha)} 
ch_s \, \overline{ \theta_\alpha M( zs \lambda + w_\lambda^0(-\rho) t)}$.

From  Theorem \ref{CCTheorem}, we see that $ch_s \, \theta_\alpha 
\overline{M( z \lambda + w_\lambda^0(-\rho )t )}$ must be a linear combination 
of $ch_s \overline{ M( z \lambda + w_\lambda^0(-\rho)t)}$ and 
of $ch_s \overline{ M( zs \lambda + w_\lambda^0(-\rho)t)}$.  The discussion 
on the bottom of p. 184 and the top of p. 185 of \cite{Y} shows that 
for $z < zs$:
\begin{eqnarray*}
 ch_s \, \theta_\alpha \overline{M( z \lambda + w_\lambda^0( t) )} &=&
(-1)^{\epsilon(\lambda_\alpha^+ - z \nu_\alpha)}\sgn( D_{F(\nu_\alpha)}(z 
\nu_\alpha) \bar{a}_{z\nu_\alpha}'' \bar{c}_z')  \,
ch_s \, \overline{M( z \lambda  + w_\lambda^0(-\rho) t)} \\
&&+ (-1)^{\epsilon(\lambda_\alpha^+ - zs \nu_\alpha)} \sgn( D_{F(\nu_\alpha)}
(zs \nu_\alpha) \bar{a}_{zs \nu_\alpha}'' \bar{c}_z') \,
 ch_s \, \overline{M( zs \lambda + w_\lambda^0(
 -\rho ) t )}.
\end{eqnarray*}
Using Lemma 4.6.7 and Lemma 4.6.9 (ii) and (iii) of \cite{Y} with $\delta 
= w_\lambda^0(-\rho)$ gives values for the signs, and we find that
\begin{eqnarray*}
ch_s \, \overline{ \theta_\alpha M( z \lambda + w_\lambda^0(-\rho) t)} &=&
(-1)^{\epsilon(\lambda_\alpha^+ + \lambda_\alpha^- + z( \lambda, \alpha^\vee)
\alpha)} ch_s \,\overline{M( zs \lambda + w_\lambda^0(-\rho) t)} \\
&& - (-1)^{\epsilon( \lambda_\alpha^+ + \lambda_\alpha^-)} ch_s \, 
\overline{M( z \lambda + w_\lambda^0(-\rho) t)}.
\end{eqnarray*}

Equation (\ref{SCoherentVanish}) follows immediately from Theorem 
\ref{CCTheorem}.

To prove (\ref{SCoherentSimple}), we note that by Lemma 4.6.4 and the 
formula before Proposition 4.6.6 of \cite{Y} (corrected to use appropriate 
notation),
\begin{eqnarray*}
ch_s \overline{ \theta_\alpha L( z \lambda + w_\lambda^0(-\rho) t )} &=& 
\sgn( \bar{c}_{zs}'' \bar{c}_z') ch_s \, L( zs \lambda ) \\
&& + \sgn( \bar{c}_z'' ( w_\lambda^0(-\rho), z\alpha^\vee) \bar{c}_z')
\sum_{y > ys} a^{\lambda,w_\lambda^0}_{w_\lambda^0 z, w_\lambda^0 y,1} ch_s 
L( y \lambda )
\end{eqnarray*}
where $c_z'' = (-1)^{\epsilon( \lambda_\alpha^+ - z \nu_\alpha )} D_{F(
\nu_\alpha)}(z\nu_\alpha) a_{z \nu_\alpha}'' ( w_\lambda^0(-\rho)t, z 
\alpha^\vee)$ and
$c_{zs}'' = (-1)^{\epsilon( \lambda_\alpha^+ - zs \nu_\alpha )} D_{F(
\nu_\alpha)}(zs\nu_\alpha) a_{zs \nu_\alpha}''$.
Using Lemmas 4.6.7 and 4.6.9 of \cite{Y} as before, we obtain the desired
result.
\end{proof}
\end{theorem}
Of course the formulas in this theorem are very simple if $\g_0$ is a 
compact real form.  We will study signature character inversion for  
compact real forms before treating the general case. 

\section{Signature Character Inversion for Compact Real Forms}
The goal of this section is to perform signature character inversion by 
induction to give a simple expression for $ch_s L( x \lambda )$ as a linear 
combination of signature characters of Verma modules, significantly 
improving Theorem 3.2.3 of \cite{Y}.

First, we list the recursive formulas for computing signed Kazhdan-Lusztig 
polynomials for compact real forms.  Substituting into Proposition 
\ref{SKLRecursive} or modifying the recursive formulas for classical 
Kazdhan-Lusztig polynomials knowing that $P_{x,y}(-q) = P_{x,y}^{\lambda, 
w_\lambda^0}(q)$ (see \cite{Y2}, Theorem 4.6), we have:
\begin{lemma}
Suppose $\g_0$ is compact.  Then the signed Kazhdan-Lusztig polynomials 
may be computed using $P_{x,x}^{\lambda,w_\lambda^0}(q) = 1$, 
$P^{\lambda,w_\lambda^0}_{x,y}(q) = 0$
if $x \not \leq y$, and the following recursive formulas where $s = s_\alpha$
where $\alpha \in \Pi_\lambda$:
\begin{itemize}
\item[a)]  If $ys > y$ and $xs > x \geq y$ then:
$$P^{\lambda,w_\lambda^0}_{w_\lambda^0 x, w_\lambda^0 y}(q) =
 P^{\lambda, w_\lambda^0}_{w_\lambda^0 xs, w_\lambda^0 y} (q).$$

\item[a')] If $sy > y$ and $sx > x \geq y$ then:
$$P^{\lambda,w_\lambda^0}_{w_\lambda^0 x, w_\lambda^0 y}(q) =  
P^{\lambda,w_\lambda^0}_{w_\lambda^0 sx, w_\lambda^0 y}(q).$$

\item[b)] If $y > ys$ and $x < xs$ then:
\begin{eqnarray*}
 - P^{\lambda, w_\lambda^0}_{ w_\lambda^0 xs, w_\lambda^0 y}(q) + q P^{\lambda,
w_\lambda^0}_{w_\lambda^0 x,
w_\lambda y} (q) &=& q \sum_{z \in W_\lambda | z < zs} a^{\lambda,
w_\lambda^0}_{w_\lambda^0 z, w_\lambda^0 y} q^{\frac{\ell(z) - \ell(y) - 1}{2}}
P^{\lambda,w_\lambda^0}_{ w_\lambda^0 x, w_\lambda^0 z}(q) \\
&& - P^{\lambda,w_\lambda^0}_{w_\lambda^0 x, w_\lambda^0 ys}(q).
\end{eqnarray*}
\end{itemize}
We could also use the recursive formulas:
\begin{itemize}
\item[a)]  If $ys < y$ and $x, xs \leq y$ then:
$$P^{\lambda,w_\lambda^0}_{x, y}(q) =
 P^{\lambda, w_\lambda^0}_{xs, y} (q).$$

\item[a')] If $sy < y$ and $sx, x \leq y$ then:
$$P^{\lambda,w_\lambda^0}_{x, y}(q) =  
P^{\lambda,w_\lambda^0}_{sx, y}(q).$$

\item[b)] If $y < ys$ and $x > xs$ then:
$$ - P^{\lambda, w_\lambda^0}_{ xs, y}(q) + q P^{\lambda,
w_\lambda^0}_{x, y} (q) = q \sum_{z \in W_\lambda | z > zs} a^{\lambda,
w_\lambda^0}_{z, y} q^{\frac{\ell(y) - \ell(z) - 1}{2}}
P^{\lambda,w_\lambda^0}_{ x, z}(q) 
 - P^{\lambda,w_\lambda^0}_{ x, ys}(q).$$
\end{itemize}
\end{lemma}

Theorem \ref{SKLCCTheorem} applied to compact real forms gives:
\begin{lemma} \label{CptCoherent}
Let $\g_0$ be compact.
Let $\lambda \in \h^*$ be antidominant regular,
$\alpha \in \Pi_\lambda$, and let $s = s_\alpha$.
\begin{enumerate}
\item If $z \in W_\lambda$ and $z < zs$, then:
\begin{eqnarray*}
ch_s \, \overline{ \theta_\alpha M( z \lambda + w_\lambda^0(-\rho) t)} &=&
- 
ch_s \, \overline{ \theta_\alpha M( zs \lambda + w_\lambda^0(-\rho) t)}  \\
&=& 
 ch_s \,\overline{M( zs \lambda + w_\lambda^0(-\rho) t)}  
-  ch_s \, \overline{M( z \lambda + w_\lambda^0(-\rho) t)}
\end{eqnarray*}
\item If $z < zs$, then
\begin{eqnarray*}
ch_s \, \overline{\theta_\alpha L( z \lambda + w_\lambda^0 (-\rho) t )} &=& 
 ch_s \, L( zs \lambda ) 
- \sum_{y > ys} a^{\lambda,w_\lambda^0}_{w_\lambda^0 z, w_\lambda^0 y, 1}
ch_s L( y \lambda ).
\end{eqnarray*}
\end{enumerate}
\end{lemma}

Our signature character inversion formula is the following:
\begin{theorem}
Let $\g_0$ be compact.
Let $\lambda \in \h^*$ be antidominant and let $x \in W_\lambda$.  Then 
inverting the formula:
$$ch_s \, \overline{M( x \lambda + w_\lambda^0(-\rho) t )} = 
\sum_{y \leq x} P^{\lambda,
w_\lambda^0}_{w_\lambda^0 x, w_\lambda^0 y}(1) ch_s L( y \lambda )$$
gives:
$$ch_s \, L( x \lambda ) = \sum_{y \leq x} (-1)^{\ell(x)-\ell(y)} \, 
P^{\lambda,w_\lambda^0}_{y,x}( 1) \, ch_s \, \overline{M(y \lambda + 
w_\lambda^0(-\rho)t)}.$$
\begin{proof}
We prove this result by induction on $x$.

If $x = 1$, then $M( \lambda )  = L( \lambda )$ while $P^{\lambda, 
w_\lambda^0}_{1,1}(q) = 1$, and the theorem holds.

Suppose that $x > 1$.  By induction, assume that for all $z < x$, 
$$ch_s \, L( z \lambda ) = \sum_{y \leq z} (-1)^{\ell(z)-\ell(y)} \, 
P^{\lambda,w_\lambda^0}_{y,z}( 1) \, ch_s \, \overline{M(y \lambda + 
w_\lambda^0(-\rho)t)}.$$  
We wish to prove the theorem for $x$.  There exists
some $s = s_\alpha$ where $\alpha \in \Pi_\lambda$, so that $\ell(x)  = 
\ell( xs ) + 1$.  By the induction hypothesis, 
$$ch_s \, \overline{L( xs \lambda + w_\lambda^0(-\rho)t )} = \sum_{y \leq xs} 
(-1)^{\ell(xs)-\ell(y)} \, P^{\lambda,w_\lambda^0}_{y,xs}( 1) \, ch_s \, 
\overline{M(y \lambda + w_\lambda^0(-\rho)t)}.$$  
We apply $\theta_\alpha$ to both sides and obtain:
$$ch_s \, \overline{ \theta_\alpha L( xs \lambda + w_\lambda^0(-\rho)t )} = 
\sum_{y \leq xs} (-1)^{\ell(xs)-\ell(y)} \, P^{\lambda,w_\lambda^0}_{y,xs}( 1) 
\, ch_s \, \overline{\theta_\alpha M(y \lambda + w_\lambda^0(-\rho)t)}.$$  
Applying Lemma \ref{CptCoherent} gives:
\begin{eqnarray*}
ch_s \, L( x \lambda ) - \sum_{y > ys} a^{\lambda,w_\lambda^0}_{w_\lambda^0 
xs, w_\lambda^0 y, 1} ch_s \, L( y \lambda ) &= &
\sum_{y \leq xs} (-1)^{\ell(xs)-\ell(y)} \, P^{\lambda,w_\lambda^0}_{y,xs}( 1) 
\times
\\
&&\varepsilon_y \left( ch_s \, \overline{M( y\lambda + w_\lambda^0( -\rho) t)} 
- ch_s \, \overline{M( ys \lambda + w_\lambda^0( -\rho)t )} \right)
\end{eqnarray*}
where $\epsilon_y = 1$ if $y > ys$ and $-1$ if $y < ys$.  Applying the 
induction hypothesis to $ch_s \, L( y \lambda )$ and rearranging, we have:
\begin{eqnarray*}
ch_s \, L( x \lambda ) &=& 
\sum_{y \leq xs} (-1)^{\ell(xs)-\ell(y)} \, P^{\lambda,w_\lambda^0}_{y,xs}( 1) 
\varepsilon_y \left( ch_s \, \overline{M( y\lambda + w_\lambda^0( -\rho) t)} 
- ch_s \, \overline{M( ys \lambda + w_\lambda^0( -\rho)t )} \right) \\
&& + \sum_{y > ys} a^{\lambda,w_\lambda^0}_{w_\lambda^0 xs, w_\lambda^0 y, 1}
\sum_{z \leq y} (-1)^{\ell(y) - \ell(z)} P^{\lambda,w_\lambda^0}_{z,y}(1) 
ch_s \overline{M( z \lambda + w_\lambda^0(-\rho)t)}
\end{eqnarray*}
We wish to show that the right side is equal to $\sum_{y \leq x} (-1)^{
\ell(x)-\ell(y)} P^{\lambda,w_\lambda^0}_{y,x}(1) ch_s \, \overline{ M( y 
\lambda + w_\lambda^0( -\rho)t )}$.  Suppose $z \leq x$.  We compute the 
coefficient of $ch_s \overline{M( z \lambda + w_\lambda^0(-\rho) t)}$ in the 
right side.

{\bf Case 1):}  If $z < zs \leq xs$: the coefficient is
\begin{eqnarray*}
&&- (-1)^{\ell(xs) - \ell(z)}P^{\lambda,w_\lambda^0}_{z,xs}(1) 
 - (-1)^{\ell(xs)-\ell(zs)} P^{\lambda,w_\lambda^0}_{zs, xs}(1)
+ \sum_{y > ys, y \geq z} a^{\lambda,w_\lambda^0}_{w_\lambda^0 xs, w_\lambda^0
y,1} (-1)^{\ell(y) - \ell(z)} P^{\lambda,w_\lambda^0}_{z,y}(1) \\
&=& (-1)^{\ell(x)-\ell(z)} \left( P^{\lambda,w_\lambda^0}_{z,xs}(1)- P^{\lambda,
w_\lambda^0}_{zs,xs}(1) + \sum_{y > ys, y \geq z} a^{\lambda,w_\lambda^0}_{
y, xs, 1} P^{\lambda,w_\lambda^0}_{z,y}(1) \right) \\
&=& (-1)^{\ell(x)-\ell(z)} \left( P^{\lambda,w_\lambda^0}_{z,xs}(1)- P^{\lambda,
w_\lambda^0}_{zs,xs}(1) + \sum_{y > ys, y \geq zs} a^{\lambda,w_\lambda^0}_{
y, xs, 1} P^{\lambda,w_\lambda^0}_{zs,y}(1) \right) \\
&=& (-1)^{\ell(x)-\ell(z)} P^{\lambda,w_\lambda^0}_{zs, x} (1) 
\quad \text{by formula b)} \\
&=& (-1)^{\ell(x) -\ell(z)} P^{\lambda,w_\lambda^0}_{z,x} (1).
\end{eqnarray*}
The first equality holds
since  $a^{\lambda,w_\lambda^0}_{w_\lambda^0 xs, w_\lambda^0 y, 1} \neq 0$ 
implies that $\ell(x)$ and $\ell(y)$ have the same parity, and since 
$P^{\lambda,w_\lambda^0}_{u,v}(q) = P_{u,v}(-q)$ for compact real 
forms implies that 
\begin{eqnarray*}
a^{\lambda,w_\lambda^0}_{w_\lambda^0 xs, w_\lambda^0 y, 1}
&=& (-1)^{\frac{\ell(xs) - \ell(y) - 1}{2}}a_{w_\lambda^0 xs, w_\lambda^0 y, 1} \\
&=& (-1)^{\frac{\ell(w_\lambda^0 y) - \ell(w_\lambda^0 xs) - 1}{2}}
a_{y,xs,1} \quad \text{by \cite{H}, p. 165} \\ 
&=& a^{\lambda,w_\lambda^0}_{y,xs,1}.
\end{eqnarray*}
The second equality holds by formula a) and since for $z < zs$ and $y > 
ys$, by Property Z, $y \geq z$ if and only if $y \geq zs$.

{\bf Case 2):}
If $zs < z \leq xs$: the coefficient of $ch_s \, \overline{M( z\lambda 
+ w_\lambda^0 (-\rho)t)}$ is
\begin{eqnarray*}
&& (-1)^{\ell(xs) - \ell(z)} P^{\lambda,w_\lambda^0}_{z,xs}(1) + 
(-1)^{\ell(xs) - \ell(zs)}P^{\lambda,w_\lambda^0}_{zs,xs}(1) + 
\sum_{y > ys} a^{\lambda,w_\lambda^0}_{w_\lambda^0 xs, w_\lambda^0 y}
(-1)^{\ell(y)-\ell(z)} P^{\lambda,w_\lambda^0}_{z,y}(1) \\
&=& (-1)^{\ell(x) - \ell(z)} \left( - P^{\lambda,w_\lambda^0}_{z,xs}(1) + 
P^{\lambda,w_\lambda^0}_{zs,xs}(1) + 
\sum_{y > ys, y \geq z} a^{\lambda,w_\lambda^0}_{y,xs,1} P^{\lambda,
w_\lambda^0}_{z,y}(1) \right)
\\
&=& (-1)^{\ell(x) - \ell(z)} P^{\lambda,w_\lambda^0}_{z,x}(1) 
\quad \text{by formula b).}
\end{eqnarray*}

{\bf Case 3):}  If $z \not \leq xs$ but $zs \leq xs$:
then $zs < z$ and by Property Z, $z \leq x$.  The coefficient of 
$ch_s \, \overline{ M(z\lambda + w_\lambda^0(-\rho)t)}$ is:
\begin{eqnarray*}
&& (-1)^{\ell(xs) - \ell(zs)} P^{\lambda,w_\lambda^0}_{zs,xs}(1) + 
\sum_{y > ys, y \geq z}
(-1)^{\ell(y) - \ell(z)} a^{\lambda,w_\lambda^0}_{w_\lambda^0 xs, 
w_\lambda^0 y, 1} P^{\lambda,w_\lambda^0}_{z,y}(1) \\
&=& (-1)^{\ell(x) - \ell(z)} \left( P^{\lambda,w_\lambda^0}_{zs,xs}(1)
+ \sum_{y > ys, y \geq z} a^{\lambda,
w_\lambda^0}_{w_\lambda^0 xs, w_\lambda^0 y, 1} P^{\lambda,w_\lambda^0}_{
z,y}(1) \right ) \\
&=& (-1)^{\ell(x) - \ell(z)} P^{\lambda,w_\lambda^0}_{zs,xs}(1) \\
&=& (-1)^{\ell(x) - \ell(z)} P^{\lambda,w_\lambda^0}_{z,x}(1).
\end{eqnarray*}
The second equality holds since $a^{\lambda,w_\lambda^0}_{w_\lambda^0 
xs, w_\lambda^0 y, 1} \neq 0$ implies that $y \leq xs$ while 
$P_{z,y} \neq 0$ implies that $z \leq y$.  There are no $y$ so that 
$z \leq y \leq xs$, so the sum is zero.

The third equality holds by, for example, formula 4.1c) of \cite{dC} and by 
Theorem 4.6 of \cite{Y2}.

{\bf Case 4):}  If $z \leq xs$ but $zs \not \leq xs$:  then $zs > z$ 
and by Property Z, $zs \leq x$.
The coefficient of 
$ch_s \, \overline{ M(z\lambda + w_\lambda^0(-\rho)t)}$ is:
\begin{eqnarray*}
&& -(-1)^{\ell(xs) - \ell(z)} P^{\lambda,w_\lambda^0}_{z,xs}(1)
+ \sum_{y > ys, y \geq z} (-1)^{\ell(y) - \ell(z)} 
a^{\lambda,w_\lambda^0}_{w_\lambda xs, w_\lambda y, 1} 
P^{\lambda,w_\lambda^0}_{z,y}(1) \\
&=& (-1)^{\ell(x)-\ell(z)} \left( P^{\lambda,w_\lambda^0}_{z,xs}(1) -
P^{\lambda,w_\lambda^0}_{zs,xs}(1)  + \sum_{y > ys, y \geq zs}
a^{\lambda,w_\lambda^0}_{y,xs,1} P^{\lambda,w_\lambda^0}_{zs,y}(1)
\right )
\quad \text{since } P^{\lambda,w_\lambda^0}_{zs,xs} = 0 \\
&=& (-1)^{\ell(x)-\ell(z)} P^{\lambda,w_\lambda^0}_{zs,x}(1)  
= (-1)^{\ell(x)-\ell(z)} P^{\lambda,w_\lambda^0}_{z,x}(1).
\end{eqnarray*}
In all cases, we see that the theorem holds for $x$.  By induction, the 
theorem holds in general.
\end{proof}
\end{theorem}

\section{Signature Character Inversion for Equal Rank Real Forms}
An easy change of basis argument extends the signature character inversion 
theorem from compact real forms to all equal rank real forms.
Recall that signed Kazhdan-Lusztig polynomials for equal rank real forms 
and classical Kazhdan-Lusztig polynomials are closely related:
\begin{proposition}(\cite{Y2}, Theorem 4.6)
Suppose $\g_0$ is an equal rank real form.
Let $\lambda$ be antidominant and let $x, y \in W_\lambda$.  Signed 
Kazhdan-Lusztig polynomials and classical Kazhdan-Lusztig polynomials 
have the following relationship:
$$P^{\lambda,w_\lambda^0}_{x,y}(q) = (-1)^{\epsilon( x \lambda - 
y \lambda )} P_{x,y}(-q)$$
where $\epsilon$ is the $\bbZ_2$-grading on the (imaginary) root lattice.  
That is, $\epsilon( \mu )$ is the parity of the number of non-compact roots 
in an expression for $\mu$ as a sum of roots.
\end{proposition}

Therefore:
\begin{theorem}
Let $\g_0$ be an equal rank real form, $\lambda \in \h^*$ be antidominant, 
and let $x \in W_\lambda$.  Then inverting the formula:
$$ch_s \, \overline{M( x \lambda + w_\lambda^0(-\rho) t )} = 
\sum_{y \leq x} P^{\lambda,
w_\lambda^0}_{w_\lambda^0 x, w_\lambda^0 y}(1) ch_s L( y \lambda )$$
gives:
$$ch_s \, L( x \lambda ) = \sum_{y \leq x} (-1)^{\ell(x)-\ell(y)} \, 
P^{\lambda,w_\lambda^0}_{y,x}( 1) \, ch_s \, \overline{M(y \lambda + 
w_\lambda^0(-\rho)t)}.$$
\begin{proof}
In the case where $\g_0$ is compact, $P^{\lambda,w_\lambda^0}_{x,y}(q) = 
P_{x,y}(-q)$.  We have shown in this case that
$$ch_s \, \overline{M( x \lambda + w_\lambda^0(-\rho) t )} = 
\sum_{y \leq x} P_{w_\lambda^0 x, w_\lambda^0 y}(-1) ch_s L( y \lambda )$$
inverts to give:
$$ch_s \, L( x \lambda ) = \sum_{y \leq x} (-1)^{\ell(x)-\ell(y)} \, 
P_{y,x}( -1) \, ch_s \, \overline{M(y \lambda + 
w_\lambda^0(-\rho)t)}.$$
Note that in the equal rank case, 
$$ch_s \, \overline{M( x \lambda + w_\lambda^0(-\rho) t )} = 
\sum_{y \leq x} (-1)^{\epsilon(x\lambda-y\lambda)} P_{w_\lambda^0 x, 
w_\lambda^0 y}(-1) ch_s L( y \lambda ),$$
so the equal rank case differs from the compact case by conjugating the 
change of basis matrix by the diagonal matrix with entries 
$(-1)^{\epsilon(x\lambda)}$ in the $(x,x)$-positions on the diagonal.  
(Extend $\epsilon$ naturally to $\bbQ \Lambda_r$.)  It thus follows that 
the inverse change of basis matrix also differs by conjugation by the 
diagonal matrix with $(-1)^{\epsilon(x\lambda)}$ on the diagonal, whence
\begin{eqnarray*}
ch_s \, L( x \lambda ) &=& \sum_{y \leq x} (-1)^{\ell(x)-\ell(y)} \, 
(-1)^{\epsilon(x\lambda-y\lambda)} P_{y,x}( -1) \, ch_s \, 
\overline{M(y \lambda + w_\lambda^0(-\rho)t)} \\
&=& \sum_{y \leq x} (-1)^{\ell(x)-\ell(y)} \, 
 P^{\lambda,w_\lambda^0}_{y,x}( 1) \, ch_s \, 
\overline{M(y \lambda + w_\lambda^0(-\rho)t)}, \\
\end{eqnarray*}
proving our theorem for equal rank real forms.
\end{proof}
\end{theorem}

\section{Conclusion}
Although the classification of unitary highest weight modules has been 
solved by work of Enright-Howe-Wallach, it would be interesting to recover the 
classification using signed Kazhdan-Lusztig polynomials and the formulas in 
this paper and in \cite{Y2}.  Cohomological induction applied to highest 
weight modules produces Harish-Chandra modules for which signatures can be 
recorded using signed Kazhdan-Lusztig-Vogan polynomials.  Techniques used to 
identify unitary representations among highest weight modules may very well 
have analogues for Harish-Chandra modules.

\bibliographystyle{alpha}
\bibliography{Inversion}
\end{document}